\documentclass[a4paper,english,leqno]{amsart}

\title[Semilinear parabolic SPDEs with polynomially bounded coefficients]{Solution theory to\\semilinear parabolic stochastic partial differential equations\\
with polynomially bounded coefficients}

\author{Alessia Ascanelli}
\address{Dipartimento di Matematica ed Informatica, Universit\`a di Ferrara, Via Machiavelli n.~30, 44121 Ferrara, Italy}
\email{alessia.ascanelli@unife.it}

\author{Sandro Coriasco}
\address{Dipartimento di Matematica ``G. Peano'', Universit\`a degli Studi di Torino, via Carlo Alberto n.~10, 10123 Torino, Italy}
\email{sandro.coriasco@unito.it}

\author{Andr{\'e} S{\"u}\ss}
\address{C/O Dipartimento di Matematica ed Informatica, Universit\`a di Ferrara, Via Machiavelli n.~30, 44121 Ferrara, Italy}
\email{suess.andre@web.de}

\date{}

\usepackage[utf8]{inputenc}
\usepackage{mathrsfs}
\usepackage{babel}
\usepackage{amsmath,amsthm,amssymb,xy,amsxtra,latexsym,verbatim}
\usepackage{hyperref}
\usepackage{enumitem}
\usepackage{pdfsync}
\usepackage{amsfonts}
\usepackage{colortbl}
\usepackage{url}

\usepackage{amsxtra}
\usepackage{pxfonts}

\newcommand*{\ii}{\mathrm{i}}

\newcommand*{\scrF}{\ensuremath{\mathscr{F}}}	
\newcommand*{\scrL}{\ensuremath{\mathscr{L}}}	

\newcommand*{\caF}{\ensuremath{\mathcal{F}}}		
\newcommand*{\caH}{\ensuremath{\mathcal{H}}}
\newcommand*{\caS}{\ensuremath{\mathcal{S}}}		

\newcommand*{\caM}{\ensuremath{\mathcal{M}}}

\def\ds{\displaystyle}

\newcommand*{\N}{\mathbb{N}}										
\newcommand*{\R}{\mathbb{R}}										
\newcommand*{\Rd}{{\mathbb{R}^d}}							
\newcommand*{\C}{\mathbb{C}}
\newcommand*{\calS}{\mathbb{S}}
				
\newcommand*{\E}{\mathbb{E}}										
\renewcommand*{\P}{\mathbb{P}}									
\renewcommand*{\Re}{\mathrm{Re}}								

\newcommand{\x}{\langle x\rangle}
\newcommand{\csi}{\langle \xi \rangle}

\newcommand{\pdd}{\langle D \rangle}
\newcommand{\jap}{\langle \cdot\rangle}
\def\<{{\langle}}
\def\>{{\rangle}}

\newcommand{\Lip}{\mathrm{Lip}}
\newcommand{\Liploc}{\mathrm{Lip}_{\mathrm{loc}}}
\newcommand{\rf}{\mathrm{rf}}
\newcommand{\fv}{\mathrm{fv}}

\newcommand{\onehalf}{1/2}

\numberwithin{equation}{section}
\allowdisplaybreaks

\theoremstyle{plain}
\newtheorem{lemma}{Lemma}[section]
\newtheorem{theorem}[lemma]{Theorem}
\newtheorem{proposition}[lemma]{Proposition}
\newtheorem{corollary}[lemma]{Corollary}

\theoremstyle{definition}
\newtheorem{definition}[lemma]{Definition}
\newtheorem{remark}[lemma]{Remark}
\newtheorem{example}[lemma]{Example}

\newtheorem*{maintheorem}{Theorem}

\newcommand{\beqsn}{\arraycolsep1.5pt\begin{eqnarray*}}
\newcommand{\eeqsn}{\end{eqnarray*}\arraycolsep5pt}
\newcommand{\beqs}{\arraycolsep1.5pt\begin{eqnarray}}
\newcommand{\eeqs}{\end{eqnarray}\arraycolsep5pt}

\usepackage{xcolor}
\definecolor{red}{rgb}{1,0,0}

\def\Op{ {\operatorname{Op}} }

\newcommand{\vvvert}{\|}

\begin{document}

\begin{abstract}
	We study function-valued solutions of a class of stochastic partial differential equations, involving operators
	with $(t,x)$-dependent, polynomially bounded coefficients. We consider semi\-li\-ne\-ar equations under suitable 
	parabolicity hypotheses. We provide conditions on the initial data
	and on the stochastic terms, namely, on the associated spectral measure, so that these mild  solutions
	exist uniquely in suitably chosen functional classes. 
\end{abstract}



\maketitle

%
\section{Introduction}\label{sec:intro}
We consider stochastic partial differential equations (SPDEs in the sequel) of the general form
\begin{equation}\label{eq:SPDE}
  Lu(t,x)=[\partial_t+A(t)]u(t,x) = \gamma(t,x,u(t,x)) + \sigma(t,x, u(t,x))\dot{\Xi}(t,x), \quad (t,x)\in[0,T]\times\R^d,
\end{equation}
where:
\begin{itemize}
\item[-] $A(t)$ is a continuous family of linear partial differential operators, that contains partial derivatives in space, with $(t,x)$-dependent coefficients, namely 
	\[
		[A(\cdot)u](t,x)=\sum_{|\alpha|\leq \mu}a_{\alpha}(t,x)\, (D_x^\alpha u)(t,x),
	\]
	where $D=-i\partial$, $\mu\geq 1$;
\item[-] the linear operator $L$ satisfies a parabolicity condition, see \eqref{roots} below;
\item[-] $\gamma$ and $\sigma$, the drift term and the diffusion coefficient, respectively, are real-valued functions, subject to certain regularity conditions;
\item[-] $\Xi$ is an $\mathcal S'(\R^d)$-valued Gaussian process, white in time and coloured in space, with correlation measure $\Gamma$ and spectral measure
$\mathfrak{M}$ (see Section \ref{sec:stochastics} for a precise definition);
\item [-]  $u$ is an unknown stochastic process, called \emph{solution} of the SPDE.
\end{itemize}
The equations \eqref{eq:SPDE} are semilinear, that is, the only possible non-linearities are on the right-hand side. This class of equations is widely studied in the literature, 
to model uncertainty both in life sciences, quantum field theory and engineering, see, for instance, \cite{H}. In particular, several stochastic heat models with random 
potentials (Anderson model, generalized Anderson model) have been recently considered, see \cite{GLO, GIP} and the references quoted therein. 
These models have applications in population dynamics, finance, aerodynamics and acoustics, and are commonly considered on the torus. However, 
global properties of the Anderson model on $\R^2$ are studied in \cite{HL}.

To give meaning to \eqref{eq:SPDE} we rewrite it formally in its corresponding integral form and look for \emph{mild solutions}, that is, stochastic processes $u(t,x)$ satisfying
\begin{equation}\label{eq:mildsolutionSPDE}
  u(t,x) = v_0(t,x)+\int_0^t\int_\Rd \Lambda(t,s,x,y)\gamma(s,y,u(s,y))dyds +\int_0^t\int_\Rd \Lambda(t,s,x,y)\sigma(s,y,u(s,y))\dot\Xi(s,y)dyds,
\end{equation}
where:
\begin{itemize}
\item[-] $v_0$ is a deterministic term, taking into account the initial condition;
\item[-] $\Lambda$ is a suitable kernel, associated with the fundamental solution of the (deterministic) partial differential equation $Lu(t)=[\partial_t+A(\cdot)]u(t)=0$; 
\item[-] the first integral in \eqref{eq:mildsolutionSPDE} is of deterministic type, while the second is a stochastic integral, and both are distributional integrals, since $\Lambda(t,s,x,y)$ is, in general, a distribution with respect to $(x,y)\in\R^{2d}$.
\end{itemize}

The kind of solution $u$ we can construct for equation \eqref{eq:SPDE} depends on the approach we employ to make sense of the stochastic integral appearing in \eqref{eq:mildsolutionSPDE}.
In the present paper we follow the Da Prato-Zabczyk approach (see \cite{dapratozabczyk}), which consists in associating an Hilbert space valued Brownian motion with the random noise. One can then define the stochastic integral as an infinite sum of It\^{o} integrals with respect to one-dimensional Brownian motions. This leads to solutions involving random functions taking values in suitable functional spaces.

The heat equation driven by a spatially homogeneous Wiener process
has been studied in \cite{pz}, where necessary and sufficient conditions for existence of a unique function-valued solution to the corresponding Cauchy problem in (exponentially) weighted Sobolev spaces are given. To our best knowledge, the most general result of existence and uniqueness of a function-valued solution to semilinear parabolic SPDEs is given in \cite{peszat}, where the author considers a stochastic heat-like equation having an elliptic second order operator with $x$-dependent coefficients instead of the Laplacian, and proves that there exists a unique solution in (exponentially) weighted Sobolev spaces under the assumption
\beqs\label{pezheat}
\exists \ell<1:\ \sup_{\eta\in\R^d}\int_{\R^d}\frac{\mathfrak{M}(d\xi)}{1+|\xi+\eta|^{2\ell}}<\infty.
\eeqs
Notice also the recent paper \cite{zhang}, where the authors consider quasilinear stochastic parabolic equations of the second order in diffusion form with $x$-dependent coefficients.

\vskip+0.2cm
In the present paper we show existence and uniqueness of a function-valued solution to a wide class of semilinear parabolic SPDEs of arbitrary order $\mu\geq 1$, with possibly unbounded coefficients, depending on $(t,x)\in[0,T]\times\R^d$, $d\geq 1$. The condition \eqref{ourheat} below, that we assume on the spectral measure, extends \eqref{pezheat} to our general setting, see Remark  \ref{remPeszgen} for more comments.
\vskip+0.2cm
More precisely, we consider the class of semilinear parabolic SPDES of the form \eqref{eq:SPDE}, with
\beqs\label{elle}
L=\partial_t + A(t),\qquad [A(\cdot)u](t,x)=\displaystyle\sum_{|\alpha|\leq \mu}a_{\alpha}(t,x)\, (D_x^\alpha u)(t,x),
\eeqs
where the $(t,x)$-dependent coefficients $a_\alpha$, defined on $[0,T]\times\Rd$, admit, at most, a polynomial growth as $|x|\to\infty$.  
That is, for arbitrary $m,\mu>0$, we assume $a_{\alpha}\in C([0,T], C^\infty(\R^d))$, $|\alpha|\le \mu$, and,
for all $\beta\in\N_0^d=(\N\cup\{0\})^d$, there exists a constant $C_{\alpha \beta}>0$ such that
\[
	|\partial^\beta_x a_{\alpha}(t,x)|\le C_{\alpha \beta} \x^{m-|\beta|},
\]
for all $(t,x)\in[0,T]\times\R^d$, where $\x:=(1+|x|^2)^{1/2}$.
The parabolicity of $L$ means that the parameter-dependent symbol $a(t,x,\xi)$ of the $SG$-operators family $A(t)$, defined here below, satisfies
\begin{equation}\label{roots}
	a(t,x,\xi):=\sum_{|\alpha|\le \mu}{a}_{\alpha}(t,x)\xi^\alpha\ge C \x^{m'} \csi^{\mu'},
\end{equation}
with $C>0$, $m\ge m'>0$, $\mu\ge\mu'>0$, that is, $a$ is $SG$-hypoelliptic. Postponing to the next Section \ref{sec:sgcalc} the precise characterization,
we give here an example.
\begin{example}
An example of a $SG$-parabolic operator $L$ is the generalized $SG$-heat operator, defined for every $m,\mu\in\N\setminus\{0\}$ by
\[
	L=\partial_t+\x^{2m}\langle D\rangle^{2\mu},
	\qquad x\in\R^d.
\]
In this case $m=m'$, $\mu=\mu'$, that is, $a$ is $SG$-elliptic. Operators of this form appear, for instance, as local representations of (modified) heat-type operators of the form
$L_{\mathfrak{g}}=\partial_t-\Delta_{\mathfrak{g}}+V$ on a manifold with ends, under an appropriate choice of the metric $\mathfrak{g}$ and of the potential $V$. Here we just sketch 
the construction, see \cite{CJT}, Example 5.21, for the details about the Laplace-Beltrami operator $\Delta_{\mathfrak{g}}$ associated with $\mathfrak{g}$, in relation with (modified)
wave operators in this context (see also \cite{CD}, Remark 8, for the analog example in relation with elliptic operators in the more general setting of scattering manifolds). 
As model of an ``end'', consider the cylinder $C=\calS^1\times(1,+\infty)$, $\calS^1$ the unit circle in $\R^2$, and put on $C$ the metric $\mathfrak{g}$ obtained by pulling back the 
metric $\mathfrak{h}$ on $\R^3$ given by $\mathfrak{h}=\frac14 \mathrm{diag}(z^2\langle z\rangle^{-2m}, z^2\langle z\rangle^{-2m}, 4\langle z\rangle^{-2m})$, $m>0$. Then, taking suitable local 
coordinates $x=(x_1,x_2)\in\R^2$, it follows that the Laplace-Beltrami operator $\Delta_\mathfrak{g}$ is given by $\Delta_\mathfrak{g}=(1+x_1^2+x_2^2)^{m}(\partial_{x_1}^2+\partial_{x_2}^2)=\langle x\rangle^{2m}\Delta$. 
Thus, if we choose the potential $V(x)=\langle x\rangle^{2m}$ in local coordinates, we obtain that
\[
	L_\mathfrak{g}=\partial_t-\Delta_\mathfrak{g}+V=\partial_t-\langle x\rangle^{2m}\Delta+\langle x\rangle^{2m}=\partial_t+\langle x\rangle^{2m} \langle D\rangle^{2}.
\]
\end{example}

We study SPDEs of the form \eqref{eq:SPDE}, \eqref{elle}, \eqref{roots}, and we derive conditions on the right-hand side terms $\gamma$ and $\sigma$, and on the spectral measure $\mathfrak{M}$ (hence, on $\dot{\Xi}$), such that there exists a function-valued (mild) solution to the corresponding Cauchy problem.

\medskip

As customary for the classes of the associated deterministic PDEs, we are interested in the present paper in both the smoothness, as well as the decay/growth  at spatial infinity of the solutions. 
Therefore the Cauchy data are going to be taken in the so-called Sobolev-Kato spaces
\begin{equation}\label{eq:skspace}
  	H^{z,\zeta}(\R^d)= \{u \in \caS^\prime(\R^{n}) \colon \|u\|_{z,\zeta}=
	\|{\jap}^z\pdd^\zeta u\|_{L^2}< \infty\}, \quad (z,\zeta)\in\R^2.
\end{equation}
The coefficients $\gamma,\sigma$ will be chosen in suitable classes of Lipschitz functions, denoted by
$\mathrm{Lip_{loc}}(z,\zeta,r,\rho)$ and defined here below.

\begin{definition}\label{def:lip}
The class $\mathrm{Lip}(z,\zeta,r,\rho)$, for given $z,\zeta,r,\rho\in\R$, $r,\rho\ge0$,
consists of all measurable functions $g:[0,T]\times\R^d\times\R\longrightarrow\C$ such that
there exists a real-valued, non negative,
$C_t=C(t)\in C([0,T])$, fulfilling the following:
\begin{itemize}
\item for every $v\in H^{z+r,\zeta+\rho}(\R^d)$, $t\in[0,T]$, we have
$\|g(t,\cdot,v)\|_{z,\zeta}\leq C(t)(1+\|v\|_{z+r,\zeta+\rho})$;
\item for every $v_1,v_2\in H^{z+r,\zeta+\rho}(R^d)$, $t\in[0,T]$, we have
$\|g(t,\cdot,v_1)-g(t,\cdot,v_2)\|_{z,\zeta}\leq C(t)\|v_1-v_2\|_{z+r,\zeta+\rho}$.
\end{itemize}
More generally, we say that $g\in\mathrm{Lip_{loc}}(z,\zeta,r,\rho)$ if the stated properties hold true for $v_1,v_2\in U$,
with $U$ a suitable open subset of $H^{n,\nu}(\R^d)$, for some $n\ge z+r$, $\nu\ge \zeta+\rho$ (typically, a sufficiently small neighbourhood of the initial data of the Cauchy problem).
\end{definition}

\begin{remark}\label{rem:lip}
Let $g:[0,T]\times\R^d\times\R\longrightarrow\R$ be measurable and
$\zeta=\rho=0$. Assume that
there exists a real-valued, non negative,
$C_t=C(t)\in C([0,T])$, satisfying
\begin{itemize}
\item for every $w\in\R$, $x\in\R^d$, $t\in[0,T]$, we have
$|g(t,x,w)|\leq C(t)(|\kappa(x)|+|w|)$, for some $\kappa \in H^{z,0}(\R^d)$, and
\item for every $w,v\in\R$, $x\in\R^d$, $t\in[0,T]$, we have  $|g(t,x,w)-g(t,x,v)|\leq C(t)|w-v|$.
\end{itemize}
Then, $g\in\mathrm{Lip}(z,0,r,0)$ for every $r\geq 0$. In fact, for some $C>0$,
	\begin{align*}
		\| g(t,\cdot,w) \|_{z,0}^2&=
		\| \jap^z g(t,\cdot,w) \|_{L^2}^2\leq C_t^2\|
		\jap^z (|\kappa|+|w|) \|_{L^2}^2
		\\
		&\le 2C_t^2(\|\kappa\|^2_{z,0}+\|w\|^2_{z,0})
		\le C^2 C_t^2(1+\|w\|_{z+r,0})^2,
	\end{align*}
	and similarly for the Lipschitz continuity with respect to the third variable, cfr. \cite{peszat}.
\end{remark}

\begin{remark}\label{rem:intpowers}
	Let $g(t,x,w)=w^n$, $n\in\N$. Then
	$g\in\mathrm{Lip_{loc}}(z,\zeta,r,\rho)$,
	when $z,r,\rho\ge0$,
	$\zeta>\frac{d}{2}$. In fact,
	when $w\in H^{z+r,\zeta+\rho}(\R^d)$ is such that $\|w\|_{z+r,\zeta+\rho}\le R$,
	$$
	\| w^n \|_{z,\zeta}\le C \| w^n \|_{nz,\zeta}\le C \|w \|_{z,\zeta}^n\le
	\widetilde{C} R^{n-1}\|w\|_{z+r,\zeta+\rho},
	$$
	for the algebra properties of the Sobolev-Kato spaces, see e.g.\ \cite[Proposition 2.2]{AC06}.
\end{remark}

The results proved in this paper
expand the theory developed in \cite{ACS19b} to the case of semilinear operators $L$ which are parabolic and whose coefficients are not uniformly bounded, and expand (in the semilinear case) the results of \cite{peszat} to the case of space-dependent coefficients with polynomial growth and of arbitrary order equations. Our main result reads as follows (see Sections
\ref{sec:sgcalc}, \ref{sec:nonlin}, and Theorem \ref{thm:linearcm} below for the precise definitions and statement).

\begin{maintheorem}\label{thm:main}
Let us consider the Cauchy problem
\begin{equation}\label{eq:cpintro}
	\begin{cases}
	Lu(t,x) = \gamma(t,x,u(t,x)) + \sigma(t,x,u(t,x))\dot{\Xi}(t,x),& (t,x)\in(0,T]\times\R^d,
 	\\
	\hspace*{1.3mm}u(0,x)=u_0(x),& x\in\R^d,
	\end{cases}
\end{equation}
for an SPDE associated with an SG-parabolic operator $L$ of the form \eqref{elle}, \eqref{roots} with $m\ge m'>0$, $\mu\ge\mu'>0$,
and $u_0\in H^{z, \zeta}(\Rd)$, $z,\zeta\in\R$.
Assume that $\gamma,\sigma\in\mathrm{Lip_{loc}}(z-\kappa m',\zeta,\kappa m',0)$ for some $\kappa\in[0,\onehalf)$, and that
         \beqs\label{ourheat}
         \exists \lambda\in[0,\onehalf):\ 	\sup_{\eta\in\Rd}\int_\Rd \frac{\mathfrak{M}(d\xi)}{\langle\xi+\eta\rangle^{2\lambda\mu'}}<\infty.
	\eeqs
%
Then, there exists a time horizon $0< T_0\leq T$ such that the Cauchy problem
\eqref{eq:cpintro} admits a unique solution $u\in L^2([0,T_0]\times\Omega, H^{z,\zeta}(\R^d))$ satisfying \eqref{eq:mildsolutionSPDE}, where the first integral is a Bochner integral, and the second integral is understood as the stochastic integral of a suitable $H^{z,\zeta}(\R^d)$-valued stochastic process with respect to the stochastic noise $\Xi$.
\end{maintheorem}
\begin{remark}\label{remPeszgen} Condition \eqref{ourheat} is consistent with the one in \cite[Remark 3.1]{peszat}, where the author finds, for parabolic equations associated with an elliptic operator of second order ($\mu=\mu'=2$), the condition \eqref{pezheat}. Indeed, the latter corresponds to the special case of \eqref{ourheat} with $\mu'=2$ and $2\lambda=\ell$.
The hypothesis \eqref{ourheat} can be understood as a {\emph {compatibility condition}} between the noise and the equation: as the order of the equation increases, we can allow for rougher stochastic noises $\Xi$.
\end{remark}
\begin{remark}We remark that, if $z\geq 0$ and $\zeta>d/2$, then $H^{z,\zeta}(\R^d)$ is embedded into the space $C_b(\R^d)$ of continuous and bounded functions on $\R^d$. Consequently, $u\in L^2([0,T_0]\times\Omega, C_b(\R^d))$, and can be pointwise evaluated at $x\in\R^d$.
\end{remark}

\medskip
We now give two examples of diffusion coefficients
$\sigma$ that we can allow in \eqref{eq:cpintro}.

\begin{example}
	$\sigma(t, x, u)=u^2$ is an admissible non-linearity
	for the equations we consider. More generally, we allow
	$\sigma(t, x, u)=u^n$, $n\in\N$, $n>2$, see Remark
	\ref{rem:intpowers}.
\end{example}

\begin{example}\label{ex:sigmanonlin}
	A right-hand side explicitly depending on $(t,x)\in[0,T]\times\R^d$
	and $u$, which
	is admissible for the equations we consider, is
	\beqs\label{expower}
		\sigma(t,x,u)=\langle x\rangle^{-\kappa m'}\cdot\widetilde{\sigma}(t,u),
	\eeqs
	where $\kappa\in [0,\onehalf)$, $m'$ in \eqref{roots} and $\widetilde{\sigma}\in \mathrm{Lip_{loc}}(z+\kappa m',\zeta,0,0)$. The function $\sigma$ in \eqref{expower} is indeed an element of $\mathrm{Lip_{loc}}(z,\zeta, \kappa m',0)$, as required in our main Theorem since, for every $w$ in a sufficiently small subset $U\subset H^{z+\kappa m',\zeta}(\Rd)$, we have
	\[||\sigma(t,\cdot,w)||_{z,\zeta}=||\tilde\sigma(t,\cdot,w)||_{z+\kappa m',\zeta}\leq C(t)\left(1+||w||_{z+\kappa m',\zeta}\right),\]
	and the verification of $||\sigma(t,\cdot,w_1)-\sigma(t,\cdot, w_2)||_{z,\zeta}\leq C(t)||w_1-w_2||_{z+\kappa m',\zeta}$ follows similarly.\\
	\indent To our best knowledge, a diffusion coefficient of the rather
	general form \eqref{expower} has been sistematically treated
	in the literature in \cite{ACS19b}, after the paper \cite{sanzvuillermot}, where, for $m=2$,
	it has been incorporated in a certain model equation by means
	of ad-hoc techniques.
\end{example}

\begin{example}
	More generally, an extension of
	the theory developed in the present paper allows for a stochastic
	term of the very general form
	$$\sigma(t,x,u,D_xu,\ldots,D_x^\alpha u),\qquad|\alpha|\leq m-1,$$
	in the right-hand side of \eqref{eq:SPDE}. The only difference consists
	in the form of the Lipschitz-continuity assumptions and the corresponding
	mapping properties.
\end{example}

The main tools for proving the existence of function-valued solutions to \eqref{eq:SPDE} will be pseudodifferential operators with symbols in the so-called $SG$ classes. Such symbol classes have been introduced in the '70s by H.O. Cordes (see, e.g. \cite{cordes}) and C. Parenti \cite{PA72} (see also R. Melrose \cite{ME}).

To construct the fundamental solution of the operator $L$ involved in \eqref{eq:SPDE}, we rely on the calculus of pseudo-differential operators of $SG$ type.
The proof of the main theorems of the paper employs such fundamental solution and a fixed point scheme in suitable functional spaces.

The strategy to prove the main theorem consists of the following steps:
\begin{enumerate}
\item\label{p:fundsol} construction of the fundamental solution of $L$ in \eqref{elle}, and then (formally) of the solution $u$ to \eqref{eq:cpintro};
\item application of a Banach fixed point scheme.
\end{enumerate}

For point \eqref{p:fundsol} we need, on one hand, to perform compositions between pseudodifferential operators, using the theory developed, e.g., in \cite{cordes}, and, on the other hand, the construction of the fundamental solution of parabolic operators in the $SG$ environment. The latter has been achieved, with precise information about the order of the pseudodifferential operator family $E(t,s)$ that defines the fundamental solution of $L$, in \cite{linearpara}.
\medskip

With the aim of giving a presentation as self-contained as possible, for the convenience of the reader, we provide
various preliminaries from the existing literature. The paper is organized as follows.
In Section \ref{sec:stochastics} we recall some notions about stochastic integration.
In Section \ref{sec:sgcalc} we first give a brief summary of the main tools, coming from microlocal analysis, that we use for the construction of the fundamental solution operator and of its kernel $\Lambda(t,s,x,y)$ (these results come mainly from \cite{cordes}). Subsequently, we recall the construction of the fundamental solution of the $SG$-parabolic operator $L$, referring to \cite{linearpara} for all the proofs.
In Section \ref{sec:nonlin} we focus on the parabolic SPDE \eqref{eq:SPDE}, \eqref{elle}, \eqref{roots}, and
prove our main theorem, under appropriate assumptions (see Theorem \ref{thm:linearcm}).
In the final Section 5 we provide a comparison, in the linear case $\gamma=\gamma(t,x), \sigma=\sigma(t,x)$, between the function-valued solution to the Cauchy problem \eqref{eq:cpintro}, obtained in Section \ref{sec:nonlin}, and the random-field solution of the corresponding linear problem, recently obtained in \cite{linearpara}, showing that the two solutions turn out to coincide.

\section*{Acknowledgement}{This research has been partially supported by the first author's INdAM-GNAMPA Project 2020.}
%
\section{Stochastic integration.}\label{sec:stochastics}

\subsection{Stochastic integration with respect to a cylindrical Wiener process.} 

\begin{definition}\label{cWp}
  Let $Q$ be a self-adjoint, nonnegative definite and bounded linear operator on a separable Hilbert space $H$. An $H$-valued stochastic process $W = \{W_t(h); h\in H, t\geq0\}$ is called a {\em cylindrical Wiener process on $H$} on the complete probability space $(\Omega,\scrF,\P)$ if the following conditions are fulfilled:
  \begin{enumerate}
    \item for any $h\in H$, $\{W_t(h); t\geq0\}$ is a one-dimensional Brownian motion with variance $t\langle Qh,h\rangle_H$;
    \item for all $s,t\geq0$ and $g,h\in H$,
    \[ \E[W_s(g)W_t(h)] = (s\wedge t)\langle Qg,h\rangle_H. \]
  \end{enumerate}
  If $Q=Id_H$, then $W$ is called a standard cylindrical Wiener process.
\end{definition}

Let $\scrF_t$ be the $\sigma$-field generated by the random variables $\{W_t(h); 0\leq s\leq t, h\in H\}$ and the $\P$-null sets. The predictable $\sigma$-field is then the $\sigma$-field in $[0,T]\times\Omega$ generated by the sets $\{(s,t]\times A, A\in\scrF_t, 0\leq s<t\leq T\}$.

We define $H_Q$ to be the completion of the Hilbert space $H$ endowed with the inner product
\[ \langle g,h\rangle_{H_Q} := \langle Qg,h\rangle_H, \]
for $g,h\in H$. In the sequel, we let $\{v_k\}_{k\in\N}$ be a complete orthonormal basis of $H_Q$. Then, the stochastic integral of a predictable, square-integrable stochastic process with values in $H_Q$, $u\in L^2([0,T]\times\Omega; H_Q)$, is defined as
\[ \int_0^t u(s)dW_s := \sum_{k\in\N} \langle u,v_k\rangle_{H_Q} dW_s(v_k). \]
In fact, the series in the right-hand side converges in $L^2(\Omega,\scrF,\P)$ and its sum does not depend on the chosen orthonormal system
$\{v_k\}_{k\in\N}$. Moreover, the It\^o isometry
\[ \E\bigg[\bigg(\int_0^t u(s)dW_s\bigg)^2\bigg] = \E\bigg[\int_0^t \|u(s)\|_{H_Q}^2 ds\bigg] \]
holds true for any $u\in L^2([0,T]\times\Omega;H_Q)$.

This notion of stochastic integral can also be extended to operator-valued integrands. Let $\caH$ be a separable Hilbert space and consider $L_2(H_Q,\caH)$, the space of Hilbert-Schmidt operators from $H_Q$ to 
$\caH$. With this we can define the space of integrable processes (with respect to $W$) as the set of $\scrF$-measureable processes in $L^2([0,T]\times\Omega;L_2(H_Q,\caH))$. Since one can identify the Hilbert-Schmidt operators in $L_2(H_Q,\caH)$ with $\caH\otimes H_Q^*$, one can define the stochastic integral for any $u\in L^2([0,T]\times\Omega;L_2(H_Q,\caH))$ coordinatewise in $\caH$. Moreover, it is possible to establish an It\^o isometry, namely,
\beqs\label{isomhilb}
\E\Bigg[\bigg\|\int_0^t u(s)dW_s\bigg\|_\caH^2\Bigg] := \int_0^t \E\big[\|u(s)\|_{L_2(H_Q,\caH)}^2\big] ds.
\eeqs

\subsection{The noise term}\label{sec:noise}
 In this paper we consider a distribution-valued
Gaussian process $\{\Xi(\phi);\; \phi\in\mathcal{C}_0^\infty(\mathbb{R}_+\times\Rd)\}$ on a complete probability space $(\Omega, \scrF, \P)$,
with mean zero and covariance functional given by
\begin{equation}
	\E[\Xi(\phi)\Xi(\psi)] = \int_0^\infty\int_\Rd \big(\phi(t)\ast\tilde{\psi}(t)\big)(x)\,\Gamma(dx) dt,
	\label{eq:correlation}
\end{equation}
where $\tilde{\psi}(t,x) := (2\pi)^{-d}\,\overline{\psi(t,-x)}$, $\ast$ is the convolution operator and $\Gamma$ is a nonnegative, nonnegative definite, tempered measure on $\Rd$ usually called \emph{correlation measure}.
Then \cite[Chapter VII, Th\'{e}or\`{e}me XVIII]{schwartz} implies that there exists a nonnegative tempered measure $\mathfrak{M}$ on $\Rd$, usually called \emph{spectral measure}, such that $\caF\Gamma = \widehat{\Gamma}=\mathfrak{M}$, where $\caF$ and $\widehat{\phantom d}$ denote the Fourier transform.
By Parseval's identity, the right-hand side of \eqref{eq:correlation} can be rewritten as
\begin{equation*}
	\E[\Xi(\phi)\Xi(\psi)] = \int_0^{\infty}\int_{\Rd}[\caF\phi(t)](\xi)\
	\cdot
	\overline{[\caF\psi(t)](\xi)}\,\mathfrak{M}(d\xi) dt.
\end{equation*}
\subsection{The Cameron-Martin space associated with the noise}\label{sec:CMspace}
To give a meaning to the stochastic integral appearing in \eqref{eq:mildsolutionSPDE}
as a stochastic integral with respect to a cylindrical Wiener process on a suitable Hilbert space, we need to understand the noise $\Xi$ in terms of a canonically associated Hilbert space $\mathcal H_\Xi$, the so-called Cameron-Martin space associated with $\Xi$.
\begin{definition}
The Cameron-Martin space of $\Xi$ is the set
\begin{equation}\label{cammart}
\mathcal H_\Xi=\{\widehat{\varphi\mathfrak{M}}\colon\varphi\in L^2_{\mathfrak{M},s}(\R^d)\},
\end{equation}
where $L^2_{\mathfrak{M},s}$ is the space of symmetric functions in $L^2_\mathfrak{M}$, i.e.
$\check{\varphi}(x)=\varphi(-x)=\varphi(x)$, $x\in\R^d$, and $\ds\int_{\R^d}|\varphi(x)|^2\,\mathfrak{M}(dx)<\infty$.
\end{definition}
Clearly, $\mathcal H_\Xi\subset\mathcal S'(\R^d)$. The Cameron-Martin space $\mathcal H_\Xi$, endowed with the inner product
\[
\langle\widehat{\varphi\mathfrak{M}},\widehat{\psi\mathfrak{M}}\rangle_{\mathcal H_\Xi}:=\langle\varphi,\psi\rangle_{L^2_{\mathfrak{M},s}}, \quad \forall\varphi,\psi\in L^2_{\mathfrak{M},s}(\R^d),
\]
and the corresponding norm \[||\widehat{\varphi\mathfrak{M}}||_{\mathcal H_\Xi}^2=||\varphi||_{L^2_{\mathfrak{M},s}}^2,\]
turns out to be a real separable Hilbert space, see \cite[Propostition 2.1]{peszat}. Thus, $\Xi$ is a cylindrical  Wiener process $W$
on $(\mathcal H_\Xi,\langle\cdot,\cdot\rangle_{\mathcal H_\Xi})$ which takes values in any Hilbert space $\caH$ such that the embedding $\caH_\Xi\hookrightarrow \caH$ is an Hilbert-Schmidt map.
%
%

\subsection{The concept of solution to \eqref{eq:SPDE}}\label{sec:rigorous_solution}
With the preparation above, we can now give the precise meaning of solution to \eqref{eq:SPDE}.
\begin{definition}\label{def:rigorous_solution}
We call \textit{(mild) function-valued solution to \eqref{eq:cpintro}} an $L^2(\Omega,\caH)$-family of random elements $u(t)$, 
satisfying the stochastic integral equation

\begin{equation}
\label{eq:rigorous_solution}
	u(t) = v_0(t) + \int_0^t E(t,s)\gamma(s,u(s))ds + \int_0^t E(t,s)\sigma(s,u(s))dW_s,
\end{equation}

for all $t\in[0,T_0]$, $x\in\R^d$, where $T_0\in(0, T]$ is a suitable time horizon, 
$v_0(t)\in\caH$ for $t\in[0,T_0]$, $E(t,s)$ is the fundamental solution to $L$, $\gamma$ and $\sigma$ are nonlinear operators defined by the Nemytskii operators associated 
with the functions $\gamma$ and $\sigma$ in \eqref{eq:SPDE}, and $W$ is the cylindrical Wiener process defined in Section \ref{sec:CMspace}.
\end{definition}

\begin{remark}\label{rem:peslem} 
\begin{enumerate}
\item[i)] The fundamental solution $E(t,s)$ to $L$, described in Theorem \ref{thm:solfond}, is a family of linear operators with Schwartz kernel $\Lambda(t,s)$, $0\le s\le t\le T_0$. 
\item[ii)] It follows that, to ensure that the stochastic integral appearing in \eqref{eq:rigorous_solution} makes sense, it will be enough to verify that 
$E(t,\cdot)\sigma(\cdot,u(\cdot))\in L^2([0,T]\times\Omega; L_2(\caH_\Xi, \caH))$,
for a suitable separable Hilbert space $\caH$.
\end{enumerate}

\end{remark}

%
%
\section{Microlocal analysis and fundamental solution to parabolic equations with polynomially bounded coefficients}\label{sec:sgcalc}

\subsection{Elements of the $SG$-calculus} In this section we recall basic definitions and facts about
the so-called $SG$-calculus of pseudodifferential operators. 
In the sequel, we will often write $A\lesssim B$ when $|A|\le c\cdot |B|$, for a suitable constant $c>0$.

Given $(m,\mu) \in \R^2$, the class $S ^{m,\mu}=S ^{m,\mu}(\R^{d})$ of $SG$ symbols of order $(m,\mu)$ consists of all $a \in C^\infty(\R^d\times\R^d)$ such that, for any multiindices $\alpha,\beta \in \N_0^d$, there exist constants $C_{\alpha\beta}>0$ such that
\begin{equation}
	\label{eq:disSG}
	|D_x^{\alpha} D_\xi^{\beta} a(x, \xi)| \leq C_{\alpha\beta}
	\x^{m-|\alpha|}\csi^{\mu-|\beta|},
	\qquad (x, \xi) \in \R^d \times \R^d.
\end{equation}
For $m,\mu\in\R$, $S^{m,\mu}$ is a Fr\'echet space with the following family of seminorms:
\begin{equation}\label{seminorms}
	\vvvert a \vvvert^{m,\mu}_\ell
	=
	\max_{|\alpha+\beta|\le \ell}\sup_{x,\xi\in\R^d}\x^{-m+|\alpha|}
	                                                                     \csi^{-\mu+|\beta|}
	                                                                    | \partial^\alpha_x\partial^\beta_\xi a(x,\xi)|, \quad \ell\in\N_0,\ a\in S^{m,\mu}.
\end{equation}
The class of pseudodifferential operators $\Op (S ^{m,\mu})=\Op (S ^{m,\mu}(\R^d))$ with symbol in $S ^{m,\mu}$  is given by
\begin{equation}\label{eq:psidos}
	(\Op(a)u)(x)=(a(.,D)u)(x)=(2\pi)^{-d}\int e^{\ii x\xi}a(x,\xi)\hat{u}(\xi)d\xi, \quad a\in S^{m,\mu}(\R^d),u\in\caS(\R^d),
\end{equation}
extended by duality to $\caS^\prime(\R^d)$.
The operators in \eqref{eq:psidos} form a
graded algebra with respect to composition, i.e.,
$$
\Op (S ^{m_1,\mu _1})\circ \Op (S ^{m_2,\mu _2})
\subseteq \Op (S ^{m_1+m_2,\mu _1+\mu _2}).
$$
The symbol $c\in S ^{m_1+m_2,\mu _1+\mu _2}$ of the composed operator $\Op(a)\circ\Op(b)$,
$a\in S ^{m_1,\mu _1}$, $b\in S ^{m_2,\mu _2}$, admits the asymptotic expansion
\begin{equation}
	\label{eq:comp}
	c(x,\xi)\sim \sum_{\alpha}\frac{i^{|\alpha|}}{\alpha!}\,D^\alpha_\xi a(x,\xi)\, D^\alpha_x b(x,\xi),
\end{equation}
which implies that the symbol $c$ equals $a\cdot b$ modulo $S ^{m_1+m_2-1,\mu _1+\mu _2-1}$.
The residual elements of the calculus are operators with symbols in
\[
	 S ^{-\infty,-\infty}=S ^{-\infty,-\infty}(\R^{d})= \bigcap_{(m,\mu) \in \R^2} S ^{m,\mu} (\R^{d})
	 =\caS(\R^{2d}),
\]
that is, those having kernel in $\caS(\R^{2d})$, continuously
mapping  $\caS^\prime(\R^d)$ to $\caS(\R^d)$. For any $a\in S^{m,\mu}$, $(m,\mu)\in\R^2$,
$\Op(a)$ is a linear continuous operator from $\caS(\R^d)$ to itself, extending to a linear
continuous operator from $\caS^\prime(\R^d)$ to itself, and from
$H^{z,\zeta}(\R^d)$ to $H^{z-m,\zeta-\mu}(\R^d)$,
where $H^{z,\zeta}(\R^d)$,
$(z,\zeta) \in \R^2$, denotes the Sobolev-Kato (or \textit{weighted Sobolev}) space defined in \eqref{eq:skspace}
with the naturally induced Hilbert norm. When $z\ge z^\prime$ and $\zeta\ge\zeta^\prime$, the continuous embedding
$H^{z,\zeta}\hookrightarrow H^{z^\prime,\zeta^\prime}$ holds true, and it is compact when $z>z^\prime$ and $\zeta>\zeta^\prime$.

Moreover
\begin{equation}\label{eq:spdecomp}
	\bigcap_{z,\zeta\in\R}H^{z,\zeta}(\R^d)=H^{\infty,\infty}(\R^d)=\caS(\R^d),
	\quad
	\bigcup_{z,\zeta\in\R}H^{z,\zeta}(\R^d)=H^{-\infty,-\infty}(\R^d)=\caS^\prime(\R^d).
\end{equation}

The continuity property of
the elements of $\Op(S^{m,\mu})$ on the scale of spaces $H^{z,\zeta}(\R^d)$, $(m,\mu),(z,\zeta)\in\R^2$, is expressed
more precisely in the next Theorem \ref{thm:sobcont}.
\begin{theorem}\label{thm:sobcont}
	Let $a\in S^{m,\mu}(\R^d)$, $(m,\mu)\in\R^2$. Then, for any $(z,\zeta)\in\R^2$,
	$\Op(a)\in\scrL(H^{z,\zeta}(\R^d),H^{z-m,\zeta-\mu}(\R^d))$, and there exists a constant $C>0$,
	depending only on $d,m,\mu,z,\zeta$, such that
	\begin{equation}\label{eq:normsob}
		\|\Op(a)\|_{\scrL(H^{z,\zeta}(\R^d), H^{z-m,\zeta-\mu}(\R^d))}\le
		C\vvvert a \vvvert_{\left[\frac{d}{2}\right]+1}^{m,\mu},
	\end{equation}
	where $[t]$ denotes the integer part of $t\in\R$ and $\scrL(X, Y)$ stands for the space of linear and continuous maps from a space $X$ to a space $Y$.
\end{theorem}
An operator $A=\Op(a)$ and its symbol $a\in S ^{m,\mu}$ are called \emph{elliptic}
(or $S ^{m,\mu}$-\emph{elliptic}) if there exists $R\ge0$ such that
\[
	C\x^{m} \csi^{\mu}\le |a(x,\xi)|,\qquad
	|x|+|\xi|\ge R,
\]
for some constant $C>0$. If $R=0$, $a^{-1}$ is everywhere well-defined and smooth, and $a^{-1}\in S ^{-m,-\mu}$.
If $R>0$, then $a^{-1}$ can be extended to the whole of $\R^{2d}$ so that the extension $\tilde{a}_{-1}$ satisfies $\tilde{a}_{-1}\in S ^{-m,-\mu}$.
An elliptic $SG$ operator $A \in \Op (S ^{m,\mu})$ admits a
parametrix $A_{-1}\in \Op (S ^{-m,-\mu})$ such that
\[
A_{-1}A=I + R_1, \quad AA_{-1}= I+ R_2,
\]
for suitable $R_1, R_2\in\Op(S^{-\infty,-\infty})$, where $I$ denotes the identity operator.
In such a case, $A$ turns out to be a Fredholm
operator on the scale of functional spaces $H^{z,\zeta}(\R^d)$,
$(z,\zeta)\in\R^2$.

%

\subsection{Fundamental solution of $SG$-parabolic operators}
In this section we consider operators of the form
\beqs\label{elle2}
	L=\partial_t+A(t)=\partial_t+\Op(a(t)),
\eeqs
where, for $m,\mu>0$, $A(t)=\Op(a(t))$ are $SG$ pseudodifferential operators with parameter-dependent symbol $a\in C([0,T], S^{m,\mu}(\R^d))$.
Notice that, of course, \eqref{elle} is a special case of \eqref{elle2}, so we are considering a class of operators with more general symbols than the (polynomial) ones appearing in \eqref{elle}.

The parabolicity condition on $L$ is here expressed by means of the ($SG$-)hypoellipticity of $A(t)$, namely,
\begin{equation}\label{eq:sghypoell}
	\begin{aligned}
	&\exists C>0\; \phantom{_{\alpha\beta}}\Re \, a(t,x,\xi)\ge C \x^{m^\prime} \csi^{\mu^\prime},
	\\
	\forall\alpha,\beta\in\N^d\;&\exists C_{\alpha\beta}>0\;
	\left|\frac{\partial^\alpha_x\partial^\beta_\xi a(t,x,\xi)}{\Re \, a(t,x,\xi)}\right|\le C_{\alpha\beta} \x^{-|\alpha|}\csi^{-|\beta|}.
	\end{aligned}
\end{equation}
where  $0<m^\prime\le m$, $0<\mu^\prime\le \mu$, $t\in[0,T]$, $x,\xi\in\R^d$. $A(t)$ is ($SG$-)elliptic
if $m=m^\prime$, $\mu=\mu^\prime$, see above. Elements of the microlocal analysis
of $SG$-parabolic operators can be found in \cite{cordes, MaPa}. As customary, $A(t)$, $t\in[0,T]$,
is considered as an unbounded operator in $L^2$ with dense domain $H^{m,\mu}$ (see \cite[Ch. 3, Sec. 3-4]{cordes};
see also \cite{MaPa} for the spectral theory of corresponding self-adjoint elliptic operators).

\begin{definition}\label{def:lpara}
	We say that $L=\partial_t + \Op(a(t))$, $a\in C([0,T], S^{m,\mu}(\R^d))$ is ($SG$-)parabolic, with respect to
	$m,m^\prime,\mu,\mu^\prime$, $0<m^\prime\le m$, $0<\mu^\prime\le\mu$, if $a$ satisfies the ($SG$-)hypoellipticity
	condition \eqref{eq:sghypoell}.
\end{definition}

The results stated in the present section come from \cite{linearpara}, where all proofs can be found. Here we recall the construction and the properties
of the fundamental solution to the operator $L$ in \eqref{elle2} under assumption \eqref{eq:sghypoell}, since they are necessary to prove our main result.

\begin{theorem}\label{thm:solfond}
	Let $L=\partial_t + \Op(a(t))$, $a\in C([0,T], S^{m,\mu}(\R^d))$ be
	($SG$-)parabolic, with respect to
	$m,m^\prime,\mu,\mu^\prime$, $0<m^\prime\le m$, $0<\mu^\prime\le\mu$.
	Then, $L$ admits a fundamental solution operator $E(t,s)$,
	$0\le s\le t\le T$, $0\le s<T$, that is, an operator family $E(t,s)=\Op(e(t,s))$ with
	$e(\cdot,s)\in C((s,T],S^{0,0}(\Rd)) \cap C^1((s,T],S^{m,\mu}(\Rd))$,
	with the following properties:
	\begin{enumerate}
		\item $E$ satisfies the equation $LE(t,s)=0, \quad 0\le s<t\le T;$
		%
		%
		\item the symbol family $e(t,s)$ satisfies $e(t,s,x,\xi)\rightarrow 1 \text{ weakly in $S^{0,0}(\Rd)$ for $t\to s^+$};$
		%
		%
		\item writing $e(t,s)$ as
		\begin{equation}\label{eq:uexp}
			e(t,s,x,\xi)=\exp\left(-\int_s^t a(\tau,x,\xi)\,d\tau \right)+r_0(t,s,x,\xi),
		\end{equation}
		the symbol family $r_0(t,s)$ satisfies
		\begin{align}\nonumber
			& r_0(t,s,x,\xi)\rightarrow0 \text{ weakly in $S^{-1,-1}(\Rd)$ for $t\to s^+$},
			\\\nonumber
			&\left\{\frac{r_0(t,s,x,\xi)}{t-s}\right\}_{0\le s< t\le T} \text{ is a bounded set in $S^{m-1,\mu-1}(\Rd)$}.
		\end{align}
	\end{enumerate}
\end{theorem}
\begin{remark}
	 It is enough that \eqref{eq:sghypoell} is satisfied for $|x|+|\xi|\ge R >0$. In fact, if this is the case, there exists $M>0$ such
	that $a_M(t,x,\xi)=a(t,x,\xi)+M$ satisfies \eqref{eq:sghypoell} everywhere. Let then $E_M(t,s)$ be the fundamental solution
	of $L_M=\partial_t + \Op(a_M(t))$. Then, $E(t,s)=e^{M(t-s)}E_M(t,s)$ is the fundamental solution of $L$ and
	\[
		e^{M(t-s)}e^{-\int_s^t[a(\tau)+M]\,d\tau}=e^{-\int_s^t a(\tau)\,d\tau},
	\]
	so $E(t,s)$ has the properties stated in Theorem \ref{thm:solfond}.
\end{remark}
By the construction performed in \cite{linearpara}, $e(t,s)$ and $E(t,s)$ are continuous also with respect to $s$, $0\le s\le t \le T$.
An immediate consequence of Theorem \ref{thm:solfond}, by a Duhamel's argument and the properties of the fundamental solution
$E$, is stated in the next Theorem \ref{thm:CPsol}.
\begin{theorem}\label{thm:CPsol}
	Let $u_0\in H^{z,\zeta}(\R^d)$, $f\in C([0,T], H^{z,\zeta}(\Rd))$, $z,\zeta\in\R$, and $L=\partial_t+A(t)$ satisfy the same
	assumptions as in Theorem \ref{thm:solfond}. Then, the Cauchy problem
	\beqs\begin{cases}\label{eq:cplin}
		Lu(t,x) = f(t,x), & (t,x)\in(s,T]\times\R^d,
		 \\
		\hspace*{1.3mm}u(s,x)=u_0(x),& x\in\R^d, s\in[0,T),
		\end{cases}
	\eeqs
	admits a solution given by
	\begin{equation}\label{eq:duhamel}
		u(t,x)=E(t,s)u_0(x)+\int_s^t E(t,\tau) \, f(\tau,x) \, d\tau, \quad s\le t\le T,
	\end{equation}
	with $E(t,s)$ the fundamental solution operator obtained in Theorem \ref{thm:solfond}. Moreover,
	such solution satisfies
	\[
		u\in C([s,T],H^{z,\zeta}(\Rd))\cap C^1([s,T],H^{z-m,\zeta-\mu}(\Rd)).
	\]
\end{theorem}
In the construction of the fundamental solution $E(t,s)$, the asymptotic development $$e(t,s,x,\xi)\sim\sum_{j\geq 0}e_j(t,s,x,\xi),\qquad e_0(t,s,x,\xi)=\exp\left(-\int_s^t a(\tau,x,\xi)\,d\tau \right)$$
of the symbol $e(t,s)$ is given, and the principal part satisfies
\beqs\label{ordine}
	|e_0(t,s,x,\xi)|&=&e^{-\int_s^t \Re\, a(\tau,x,\xi)\,d\tau}\leq e^{-C(t-s)\x^{m'}\csi^{\mu'}}.
\eeqs
The next lemma shows that, actually,
for $0\le s < t\le T$, $e(t,s)$ gives rise to (a $C^1$ family of) operators of order $(-\infty,-\infty)$. This, of course, cannot
be extended by continuity up to $t=s$, but some $L^1$ regularity with respect to $t$, that will be crucial in
Section \ref{sec:nonlin}, can be achieved.

\begin{lemma}\label{lem:ref}
For every $j\in\N$, $\alpha,\beta\in\N^d$, we have, for suitable constants $C'_{j\alpha\beta}>0$,
\beqs\label{stimej}
|\partial_x^\alpha\partial_\xi^\beta e_j(t,s,x,\xi)|\leq C'_{j\alpha\beta} \sqrt{|e_0(t,s,x,\xi)|}\,\x^{-j-|\alpha|}\csi^{-j-|\beta|},
\quad 0\leq s\le  t\leq T, \; (x,\xi)\in\R^d.
\eeqs
Moreover, for every $j\in\N$, $0\leq s<  T$, $e_j(\cdot,s)\in C^1((s,T],\caS(\R^{2d}))$,
and for every $l,\lambda\in[0,1)$ we have $e(\cdot,s)\in L^1([s,T],S^{-lm',-\lambda\mu'}(\R^d))$ , $\partial_t e(\cdot,s) \in L^1([s,T],S^{m-lm',\mu-\lambda\mu'}(\R^d))$.
\end{lemma}

Lemma \ref{lem:ref} is a slight modification of Lemma 5 in \cite{linearpara}, see \cite{linearpara} for the proof. Here we only point out the difference with respect to that lemma, that is the following essential estimate on $e_0(t,s)$ (a refinement of the one given in \cite{linearpara}): by \eqref{ordine} we see that for every $m',\mu'>0$ and for every $l,\lambda\ge0$,
\beqs\label{ordinebis}|e_0(t,s,x,\xi)|&\leq& e^{-C(t-s)\x^{m'}\csi^{\mu'}}\x^{lm'}\csi^{\lambda\mu'}\x^{-lm'}\csi^{-\lambda\mu'}
\\\nonumber
&\lesssim &\frac{\ell^\ell e^{-\ell}}{C^\ell(t-s)^\ell}\x^{-lm'}\csi^{-\lambda \mu'} \lesssim (t-s)^{-\ell}\x^{-lm'}\csi^{-\lambda \mu'}, \quad \ell=\max\{l,\lambda\}.
\eeqs
\begin{corollary}\label{cor:cp}
	Under the same hypothesis of Theorem \ref{thm:CPsol}, the solution of the Cauchy problem \eqref{eq:cplin} described there satisfies,
	for any $l,\lambda\in[0,1)$,
	\[
		u\in C([s,T],H^{z,\zeta}(\Rd))\cap C^1([s,T],H^{z-m,\zeta-\mu}(\Rd)) \cap C^1((s,T],\caS(\R^d))\cap L^1([s,T],H^{z+lm',\zeta+\lambda\mu'}(\Rd)).
	\]
	It also satisfies $\partial_t u\in L^1([s,T],H^{z-m+lm',\zeta-\mu+\lambda\mu'}(\Rd))$, $l,\lambda\in[0,1)$.
\end{corollary}
\begin{remark}
Microlocal techniques, also involving Fourier integral operators in the construction of the fundamental solution, have recently been applied to obtain either function-valued solutions or random-field solutions for hyperbolic SPDEs with $(t,x)$-dependent coefficients, see \cite{ACS19b, ACS19a, alessiandre}. To our best knowledge, those are the first times that the full potential of microlocal analysis has been rigorously applied within the theory of hyperbolic SPDEs. In this paper we see that an analogous
approach can be employed as well for semilinear parabolic SPDEs (see also \cite{linearpara}).
\end{remark}

\vskip+0.2cm

%
\section{function-valued solutions for semilinear SPDEs.}\label{sec:nonlin}

In this section we state and prove our main result, Theorem \ref{thm:linearcm}, about the existence of a function-valued solution \eqref{eq:rigorous_solution} of an SPDE \eqref{eq:SPDE}
with $L$ as in \eqref{elle2},\eqref{eq:sghypoell}, under the assumptions of ($SG$-)parabolicity for the operator $L$, see Definition \ref{def:lpara}. We consider
the Cauchy problem
\beqs\label{cpsemilin2}\begin{cases}
Lu(t,x) = f(t,x,u(t,x)) = \gamma(t,x,u(t,x)) + \sigma(t,x,u(t,x))\dot{\Xi}(t,x),& (t,x)\in(0,T]\times\R^d,
 \\
\hspace*{1.3mm}u(0,x)=u_0(x),& x\in\R^d,
\end{cases}
\eeqs
with the aim of finding conditions on $L$, on the stochastic noise $\dot \Xi$, and on $\sigma, \gamma, u_0$,
such that \eqref{cpsemilin2} admits a function-valued solution. The conditions on the stochastic noise will be given on the spectral measure
$\mathfrak{M}$ corresponding to the correlation measure $\Gamma$ related to the noise $\dot \Xi$; the right hand side coefficients $\gamma,\sigma$ will be taken in suitable $\mathrm{Lip_{loc}}(z,\zeta,r,\rho)$ classes, see Definition \ref{def:lip}.

To prove Theorem \ref{thm:linearcm}, we will rely on the properties of the fundamental solution to $L$, recalled in the previous Section \ref{sec:sgcalc}, and of the Cameron-Martin space associated with the noise $\Xi$, 
described in Section \ref{sec:stochastics}.
The following Lemma \ref{lem:weightedpesz2}  is the key result to prove the main theorem of this paper. It shows that, under suitable assumptions on $\sigma$,
the multiplication operator $\caH_\Xi\ni\psi\mapsto E(t,s)\sigma(s,\cdot,w)\cdot \psi$ is Hilbert-Schmidt from $\caH_\Xi$ to $\caH=H^{z,\zeta}$. Therefore, 
the second integral in \eqref{eq:rigorous_solution} is well-defined as stochastic integral with respect to a cylindrical  Wiener process on 
$(\mathcal H_\Xi,\langle\cdot,\cdot\rangle_{\mathcal H_\Xi})$ which takes values in
$\caH=H^{z,\zeta}$, see Remark \ref{rem:peslem}. 

\begin{lemma}\label{lem:weightedpesz2}
Let $\sigma\in \Lip(z-\kappa m',\zeta,\kappa m',0)$ for some $\kappa\ge0$, and let $E(t,s)$ be a continuous family of $SG$-pseudodifferential operators, parametrized by $0\le s\le t \le T_0$, $T_0>0$, satisfying Lemma \ref{lem:ref}. If  for some $\lambda\geq 0$ the spectral measure satisfies
\beqs\label{eqqua}
\sup_{\eta\in\R^d}\ds\int_{\R^d}\frac{\mathfrak{M}(d\xi)}{\langle\xi+\eta\rangle^{2\lambda\mu'}}<\infty,
\eeqs
then, for every $w\in H^{z,\zeta}(\R^d)$, the operator $$\Phi(t,s)=\Phi_{m', \mu', \kappa,\sigma,w}(t,s)\colon
\psi\mapsto E(t,s)\sigma(s,\cdot,w) \psi$$ belongs to $L_2(\mathcal H_\Xi, H^{z,\zeta}(\R^d))$.
Moreover, the Hilbert-Schmidt norm of $\Phi(t,s)$ can be estimated by
\[
\|\Phi(t,s)\|_{L_2(\mathcal H_\Xi, H^{z,\zeta})}^2\lesssim
C^2_{t,s}(1+\|w\|_{z,\zeta})^2
\sup_{\eta\in\Rd}\int_\Rd \frac{\mathfrak{M}(d\xi)}{\langle\xi+\eta\rangle^{2\lambda\mu'}},
\]
for $C_{t,s}=(t-s)^{-\ell}C_s$, where $\ell=\max\{k,\lambda\}$ and $C_s$ is the constant in Definition \ref{def:lip}.
\end{lemma}

\begin{proof}
Since \eqref{ordinebis} holds for every $l,\lambda\geq 0$, we choose there $l=\kappa$ and $\lambda$ from \eqref{eqqua}. Let us fix an orthonormal basis $\{e_j\}_{j\in\N}=\{\widehat{f_j\mathfrak{M}}\}_{j\in\N}$ of $\mathcal H_\Xi$, where $\{f_j\}_{j\in\N}$ is an orthonormal basis in $L^2_{\mathfrak{M},s}$. We compute
\beqs\nonumber
||\Phi(t,s)||_{L_2(\mathcal H_\Xi, H^{z,\zeta})}^{2}&=&\sum_{j\in\N}||E(t,s)\sigma(s, \cdot, w)\widehat{f_j\mathfrak{M}}||_{H^{z,\zeta}}^2
\\\nonumber
&=&\sum_{j\in\N}||\langle D\rangle^{-\lambda\mu'}\langle D\rangle^{\lambda\mu'}\langle \cdot\rangle^{z}\langle D\rangle^\zeta E(t,s)\sigma(s, \cdot, w)\widehat{f_j\mathfrak{M}}||_{L^2}^2
\\\nonumber
&=&\sum_{j\in\N}||\langle D\rangle^{-\lambda\mu'}\widetilde E(t,s)\sigma(s, \cdot, w)\widehat{f_j\mathfrak{M}}||_{L^2}^2
\\\label{las}
&=&(2\pi)^{-d}\sum_{j\in\N}\int_{\R^d}\langle\xi\rangle^{-2\lambda\mu'}\left\vert\mathcal F\left(
\widetilde E(t,s)\sigma(s, \cdot, w)\widehat{f_j\mathfrak{M}}\right)\right\vert^2(\xi)d\xi
\eeqs
with $\widetilde E(t,s)=\langle D\rangle^{\lambda\mu'}\langle \cdot\rangle^{z}\langle D\rangle^\zeta E(t,s)$ a family
of $SG$ pseudodifferential operators of order $(z-\kappa m',\zeta)$.
Now, using the well-known fact that the Fourier transform of a product is the ($(2\pi)^{-d}$ multiple of the)
convolution of the Fourier transforms, the
property $f_j(-x)=f_j(x)$ (by the definition of $L^2_{\mathfrak{M},s}$), that
$\{f_j\}$ is an orthonormal system in $L^2_{\mathfrak{M}}$, and Bessel's inequality,
we get
\beqsn
(2\pi)^{-d}\sum_{j\in\N}&&\hskip-0.3cm\left\vert\mathcal F\left(\widetilde E(t,s)\sigma(s, \cdot, w)\widehat{f_j\mathfrak{M}}\right)\right\vert^2(\xi)
\\
&=&(2\pi)^{-2d}\sum_{j\in\N}|\mathcal F\left(\widetilde E(t,s)\sigma(s, \cdot, w)\right)\ast\widehat{\widehat{f_j\mathfrak{M}}}|^2(\xi)
\\
&=&(2\pi)^{-d}\sum_{j\in\N}|\mathcal F\left(\widetilde E(t,s)\sigma(s, \cdot, w)\right)\ast f_j\mathfrak{M}|^2(\xi)
\\
&=&(2\pi)^{-d}\sum_{j\in\N}\left\vert\int_{\R^d}\left[\mathcal F\left(\widetilde E(t,s)\sigma(s, \cdot, w)\right)\right](\xi-\eta) f_j(\eta)\mathfrak{M}(d\eta)\right\vert^2
\\
&\le &(2\pi)^{-d}\int_{\R^d}\left\vert\mathcal F\left(\widetilde E(t,s)\sigma(s, \cdot, w)\right)\right\vert^2(\xi-\eta)\mathfrak{M}(d\eta).
\\
\eeqsn
Inserting this in \eqref{las}, and using the continuity of $\widetilde E(t,s)$ on the Sobolev-Kato spaces we finally get:
\beqs\nonumber
||\Phi(t,s)||_{L_2(\mathcal H_\Xi, H^{z,\zeta})}^{2}
&\leq&(2\pi)^{-d}
\int_{\R^d}\int_{\R^d}\langle\xi\rangle^{-2\lambda\mu'}\left\vert\mathcal F\left(\widetilde E(t,s)\sigma(s, \cdot, w)\right)\right\vert^2(\xi-\eta)\mathfrak{M}(d\eta)d\xi
\\\nonumber
&= &(2\pi)^{-d}\int_{\R^d}\int_{\R^d}\langle\eta+\theta\rangle^{-2\lambda\mu'}\left\vert\mathcal F\left(\widetilde E(t,s)\sigma(s, \cdot, w)\right)\right\vert^2(\theta)\mathfrak{M}(d\eta)d\theta
\\\nonumber
&\leq &(2\pi)^{-d}\left(\sup_{\theta\in\R^d}\int_{\R^d}\langle \theta+\eta\rangle^{-2\lambda\mu'}\mathfrak{M}(d\eta)\right)\int_{\R^d}\left\vert\mathcal F\left(\widetilde E(t,s)\sigma(s, \cdot, w)\right)\right\vert^2(\theta)d\theta
\\\label{battezzata2}
&=&(2\pi)^{-d}\left(\sup_{\theta\in\R^d}\int_{\R^d}\langle \theta+\eta\rangle^{-2\lambda\mu'}\mathfrak{M}(d\eta)\right)\|\mathcal F(\widetilde E(t,s)\sigma(s, \cdot, w))\|_{L^2}^2
\\\nonumber
&\lesssim &\left(\sup_{\theta\in\R^d}\int_{\R^d}\langle \theta+\eta\rangle^{-2\lambda\mu'}\mathfrak{M}(d\eta)\right)C_{t,s}^2\|\sigma(s, \cdot, w)\|_{z-\kappa m',\zeta}^2
\\\nonumber
&\leq &\left(\sup_{\theta\in\R^d}\int_{\R^d}\langle \theta+\eta\rangle^{-2\lambda\mu'}\mathfrak{M}(d\eta)\right)C_{t,s}^2C_s^2\left(1+\|w\|_{z,\zeta}\right)^2,
\eeqs
where
\beqs\label{ell}
C_{t,s}=(t-s)^{-\ell},\quad \ell=\max\{\kappa, \lambda\}
\eeqs
comes from the norm in $\scrL(H^{z-\kappa m',\zeta},H^{0,0})$ of the $SG$ operator $\widetilde E(t,s)$, with symbol of order $(z-\kappa m',\zeta)$, estimated using \eqref{eq:normsob} and \eqref{ordinebis}. Since $\sigma\in \Lip(z-\kappa m',\zeta,\kappa m',0)$, $C_s$ is the constant in Definition \ref{def:lip} and is a continuous function with respect to $s\in [0,T]$.
\end{proof}

%

We are now ready to prove the main result of this paper.
%
\begin{theorem}\label{thm:linearcm}
Let us consider the Cauchy problem \eqref{cpsemilin2} for an SPDE associated with a $SG$-parabolic operator $L$ of the form \eqref{elle2}, \eqref{eq:sghypoell},
with $u_0\in H^{z,\zeta}(\R^d)$, $z,\zeta\in\R$.
%
	%
	Assume that, for some $\kappa\in [0,\onehalf)$, $\gamma,\sigma\in\mathrm{Lip_{loc}}(z-\kappa m',\zeta,\kappa m',0)$
	in some open subset  $U\subset H^{z,\zeta}(\R^d)\hookrightarrow H^{z-\kappa m',\zeta}(\R^d)$ with $u_0 \in U$, and
	\begin{equation}\label{eq:meascm2}
  		\exists\lambda\in[0,1/2):\ \sup_{\eta\in\R^d}\int_\Rd \frac{\mathfrak{M}(d\xi)}{\langle\xi+\eta\rangle^{2\lambda\mu'}} < \infty.
	\end{equation}
	%
%
Then, there exists a time horizon $0< T_0\leq T$ such that the Cauchy problem \eqref{cpsemilin2} admits a unique
 solution $u\in L^2([0,T_0]\times\Omega, H^{z,\zeta}(\R^d))$ in the sense of Definition \ref{def:rigorous_solution}. That is, 
$u$ satisfies \eqref{eq:rigorous_solution} for all $t\in[0,T_0]$, $x\in\R^d$, 
%
where $E(t,s)$, fundamental solution  of $L$, has  Schwartz kernel $\Lambda(t,s)$, the first integral in \eqref{eq:rigorous_solution} is a Bochner integral, and the second integral in \eqref{eq:rigorous_solution} is understood as the stochastic integral of the $H^{z,\zeta}(\R^d)$-valued stochastic process $E(t,*)\sigma(*,\cdot,u(*,\cdot))$ with respect to the stochastic noise $\Xi$, in the sense explained in Section \ref{sec:stochastics}.
\end{theorem}

\begin{remark}
	Notice that, by Lemma \ref{lem:weightedpesz2}, the noise $\Xi$
	defines a cylindrical Wiener process on
	$(\mathcal H_\Xi(\R^d),\langle\cdot,\cdot\rangle_{\mathcal H_\Xi(\R^d)})$
	with values in $H^{z,\zeta}(\R^d)$.
\end{remark}

\begin{proof}[Proof of Theorem \ref{thm:linearcm}]

We know that the fundamental solution $E(t,s)$ fulfills \eqref{ordinebis} for every $l,\lambda\geq 0$. We choose $l=\kappa$ and $\lambda$ from \eqref{eq:meascm2}. The assumption $\kappa,\lambda\in [0,1/2)$ provides an $L^2$-regularity with respect to time in the estimate \eqref{ordinebis}, that will be crucial in the sequel of the
proof.

By Theorem \ref{thm:CPsol}, inserting $f(t,x)=\gamma(t,x) + \sigma(t,x)\dot{\Xi}(t,x)$ in \eqref{eq:duhamel}, we can formally construct the mild solution $u$ to \eqref{eq:SPDE}:
\[
	\begin{aligned}
	u(t,x)&=v_0(t,x)+\int_0^t\int_{\R^d}\Lambda(t,s,x,y)\gamma(s,y,u(s,y))\,dyds +
	\int_0^t\int_{\R^d}\Lambda(t,s,x,y)\sigma(s,y, u(s,y))\dot{\Xi}(s,y)\,dyds
	\\
	&=v_0(t,x)+v_1(t,x)+v_2(t,x),
	\end{aligned}
\]
for a suitable $T_0\in(0, T]$, where we indicated by $\Lambda(t,s)$ the Schwartz kernel of $E(t,s)$ and $v_0(t)=E(t,0)u_0$.
We notice that $v_0$ is well-defined by the general theory of $SG$-operators and the properties of $E(t,s)$.
In particular, by Theorem \ref{thm:sobcont} we see that, for any $\kappa,\lambda\in[0,1)$,
\begin{equation}\label{eq:regv0}
	v_0\in C([0,T], H^{z,\zeta})\cap C^1([0,T],H^{z-m,\zeta-\mu}) \cap C^1((0,T],\caS)\cap L^1([0,T],H^{z+\kappa m',\zeta+\lambda\mu'}).
\end{equation}
For $\kappa,\lambda\in[0,\onehalf)$ it also holds $v_0\in L^2([0,T], H^{z+\kappa m',\zeta+\lambda\mu'})\hookrightarrow
L^2([0,T], H^{z,\zeta})$, see below.

\vskip+0.2cm

We then consider the map $u\to \mathcal{T}u$ on $L^2([0,T_0]\times\Omega, H^{z,\zeta}(\R^d))$, defined as follows:
\beqs\label{calm}
\hskip-0.8cm \mathcal{T}u(t)&:=&
v_0(t)+\ds\int_0^t E(t,s)\gamma(s,\cdot,u(s))ds+\ds\int_0^t E(t,s)\sigma(s,\cdot,u(s))dW_s
\\
\nonumber
&:=&v_0(t)+\mathcal{T}_1u(t)+\mathcal{T}_2u(t), \quad t\in[0,T_0],
\eeqs
where the last integral on the right-hand side is understood as the stochastic integral of the stochastic process $E(t,*)\sigma(*,\cdot,u(\cdot))\in L^2([0,T_0]\times\Omega, H^{z,\zeta})$ with respect to the cylindrical Wiener process $\{W_t(h)\}_{t\in[0,T], h\in\mathcal H_{\Xi}}$, associa\-ted with the random noise $\Xi(t)$, which is well-defined by Lemma \ref{lem:weightedpesz2} and takes values in $H^{z,\zeta}$.
To prove that the mild solution \eqref{eq:rigorous_solution} of the Cauchy problem \eqref{cpsemilin2} is indeed well-defined,
it is enough to check that
\[
	\mathcal{T}\colon L^2([0,T_0]\times\Omega, H^{z,\zeta}(\R^d))\longrightarrow L^2([0,T_0]\times\Omega, H^{z,\zeta})
\]
is well-defined, it is Lipschitz continuous on $L^2([0,T_0]\times\Omega, H^{z,\zeta})$, and it becomes a contraction if we take $T_0$ small enough. Then, an application of Banach's fixed point Theorem will provide the existence of a unique solution $u\in L^2([0,T_0]\times\Omega, H^{z,\zeta})$ satisfying $u=\mathcal{T}u$, that is \eqref{eq:rigorous_solution}.

\vskip+0.2cm

We first check that $\mathcal{T}u\in L^2([0,T_0]\times\Omega, H^{z,\zeta})$ for every $u\in L^2([0,T_0]\times\Omega, H^{z,\zeta})$. We have:
\begin{itemize}
\item[-] by Lemma \ref{lem:ref}, see formula \eqref{ordinebis} with $l=\kappa$, $\displaystyle v_0\in L^2([0,T_0], H^{z+\kappa m',\zeta+\lambda \mu'})\hookrightarrow L^2([0,T_0]\times\Omega, H^{z,\zeta})$;
\item[-] $\mathcal{T}_1u$ is in $L^2([0,T_0]\times\Omega, H^{z,\zeta})$; indeed, $\mathcal{T}_1u(t)$ is defined as the Bochner integral on $[0,t]$ of the function
$s\mapsto E(t,s)\gamma(s,\cdot,u(s))$ with values in $L^2(\Omega, H^{z,\zeta})$; by the properties of Bochner integrals, the continuity of $E(t,s)$ on Sobolev-Kato spaces, and the fact that
$\gamma\in \Lip(z-\kappa m',\zeta,\kappa m',0)$, we see that (for $C_{t,s}$ in \eqref{ell})
\beqs\nonumber
\|\mathcal{T}_1u\|_{L^2([0,T_0]\times\Omega, H^{z,\zeta})}^2&&=\ds\E\left[\int_0^{T_0}\|\mathcal{T}_1u(t)\|_{z,\zeta}^2\,dt\right]
=
\ds\int_0^{T_0}\E\left[\,\left\|\ds\int_0^t E(t,s)[\gamma(s,\cdot,u(s))]ds\right\|_{z,\zeta}^2\,\right]dt
\\\nonumber
&&\leq\ds\int_0^{T_0}\ds\int_0^t \E\left[\,\left\|E(t,s)[\gamma(s,\cdot,u(s))]\right\|_{z,\zeta}^2\,\right]dsdt
\lesssim \ds\int_0^{T_0}\ds\int_0^t C_{t,s}^2 \,\E\left[\,\left\|\gamma(s,\cdot,u(s)))\right\|_{z-\kappa m',\zeta-\lambda\mu'}^2\,\right] dsdt
\\\label{uguale}
&&\lesssim \ds\int_0^{T_0}\ds\int_0^t C_{t,s}^2 \,\E\left[\,\left\|\gamma(s,\cdot,u(s)))\right\|_{z-\kappa m',\zeta}^2\,\right] dsdt
\leq \ds\int_0^{T_0}\ds\int_0^t C_{t,s}^2C_s^2\,\E\left[\left(1+\|u(s)\right\|_{z,\zeta})^2\right]dsdt
\\\nonumber
&& = \ds\int_0^{T_0}\left(\ds\int_s^{T_0} (t-s)^{-2\ell }dt\right)C_s^2\,\E\left[\left(1+\|u(s)\right\|_{z,\zeta})^2\right]ds
\leq \frac{T_0^{1-2\ell }}{1-2\ell }\cdot \left(\max_{0\leq s\leq T_0}C_s^2\right)
\cdot\ds\int_0^{T_0}\E\left[\left(1+\|u(s)\right\|_{z,\zeta})^2\right]ds
\\\nonumber
&&\lesssim T_0^{1-2\ell }
\int_0^{T_0} \E\left[\left(1+\left\|u(s)\right\|_{z,\zeta}^2\right)\right]ds
=C_{1,T_0}\left[T_0+\|u\|^2_{L^2([0,T_0]\times\Omega, H^{z,\zeta})}\right]<\infty
\eeqs
since $\ell \in[0,\onehalf)$ by the choices of $l,\lambda$ and $C_s$ is continuous with respect to $s$, with $C_{1,T_0}$ continuous with respect to $T_0$ and going to $0$ for $T_0\to0^+$;
\\
\item[-] $\mathcal{T}_2u$ is in $L^2([0,T_0]\times\Omega, H^{z,\zeta})$, in view of the fundamental isometry \eqref{isomhilb},
Lemma \ref{lem:weightedpesz2} and the fact that the expectation can be moved inside and outside time integrals, by Fubini's Theorem:
\beqsn
\|\mathcal{T}_2u\|_{L^2([0,T_0]\times\Omega, H^{z,\zeta})}^{2}&=&\E\left[\ds\int_0^{T_0}\|\mathcal{T}_2u(t)\|_{z,\zeta}^2dt\right]
\\
&=&\ds\int_0^{T_0}\E\left[\left\|\ds\int_0^t E(t,s)\sigma(s,\cdot,u(s))dW_s\right\|_{z,\zeta}^2\right]dt
\\
&=&\ds\int_0^{T_0}\ds\int_0^t \E\left[\left\|E(t,s)\sigma(s,\cdot,u(s))\right\|_{L_2(\mathcal H_\Xi, H^{z,\zeta})}^2\right]dsdt
\\
&\lesssim&\ds\int_0^{T_0}\ds\int_0^t \E\left[ C^2_{t,s} \left(1+\|u(s)\|_{H^{z,\zeta}}\right)^2
\sup_{\eta\in\Rd}\int_\Rd \frac{\mathfrak{M}(d\xi)}{\langle\xi+\eta\rangle^{2\lambda \mu'}}\right]dsdt
\\
&=&\left(\sup_{\eta\in\Rd}\int_\Rd \frac{\mathfrak{M}(d\xi)}{\langle\xi+\eta\rangle^{2\lambda \mu'}}\right)
\cdot\ds\int_0^{T_0}\ds\int_0^t C^2_{t,s}\E\left[ \left(1+\|u(s)\|_{H^{z,\zeta}}\right)^2\right]dsdt
\\
&\lesssim&{T_0}^{1-2\ell}\cdot\left(\sup_{\eta\in\Rd}\int_\Rd \frac{\mu(d\xi)}{\langle\xi+\eta\rangle^{2\lambda \mu'}}\right)
\left(T_0+\ds\int_0^{T_0} \E\left[\left\|u(s)\right\|_{z,\zeta}^2\right]ds\right)
\\
&=&C_{2,T_0}\left[T_0+\|u\|^2_{L^2([0,T_0]\times\Omega, H^{z,\zeta})}\right]<\infty,
\eeqsn
with $C_{2,T_0}$ continuous with respect to $T_0$ and going to $0$ for $T_0\to0^+$.
\end{itemize}

\vskip+0.2cm

Now, we show that $\mathcal T$ is a contraction for $T_0>0$ suitably small. We take $u_1,u_2\in L^2([0,T_0]\times\Omega, H^{z,\zeta})$ and compute
\beqs\nonumber
\|\mathcal{T}u_1&-&\mathcal{T}u_2\|_{L^2([0,T_0]\times\Omega, H^{z,\zeta})}^2
\\\nonumber
&\leq& 2\left(\|\mathcal{T}_1u_1-\mathcal{T}_1u_2\|_{L^2([0,T_0]\times\Omega, H^{z,\zeta})}^2+\|\mathcal{T}_2u_1-
\mathcal{T}_2u_2\|_{L^2([0,T_0]\times\Omega, H^{z,\zeta})}^2\right)
\\\label{pri2}
&=&2\ds\int_0^{T_0}\E\left[\,\left\|\ds\int_0^t E(t,s)(\gamma(s,\cdot,u_1(s))-\gamma(s,\cdot,u_2(s)))ds\right\|_{z,\zeta}^2\,\right]dt
\\\label{se2}
&+&2\ds\int_0^{T_0}\E\left[\,\left\|\ds\int_0^t E(t,s)(\sigma(s,\cdot,u_1(s))-\sigma(s,\cdot,u_2(s)))dW_s\right\|_{z,\zeta}^2\,\right]dt.
\eeqs
In the term \eqref{pri2} here above we can move the expectation and the $(z,\zeta)-$norm inside the integral with respect to $s$. Then, by continuity of $E(t,s)$ on the Sobolev-Kato spaces, Definition \ref{def:lip}, and the embedding $H^{z-\kappa m',\zeta}\hookrightarrow H^{z-\kappa m',\zeta-\lambda\mu'}$, we obtain
\beqsn%
\ds\int_0^{T_0}\E&&\left[\,\left\|\ds\int_0^t E(t,s)(\gamma(s,\cdot,u_1(s))-\gamma(s,\cdot,u_2(s)))ds\right\|_{z,\zeta}^2\,\right]dt
\\
&&\lesssim\ds\int_0^{T_0}\ds\int_0^t \E\left[\,\left\|E(t,s)(\gamma(s,\cdot,u_1(s))-\gamma(s,\cdot,u_2(s)))\right\|_{z,\zeta}^2\,\right]dsdt
\\
&&\leq \ds\int_0^{T_0}\ds\int_0^t C_{t,s}^2\,\E\left[\,\left\|\gamma(s,\cdot,u_1(s))-\gamma(s,\cdot,u_2(s))\right\|_{z-\kappa m',\zeta-\lambda\mu'}^2\,\right]dsdt
\\
&&\lesssim \ds\int_0^{T_0}\ds\int_0^t C_{t,s}^2\,\E\left[\,\left\|\gamma(s,\cdot,u_1(s))-\gamma(s,\cdot,u_2(s))\right\|_{z-\kappa m',\zeta}^2\,\right]dsdt
\\
&&\leq \ds\int_0^{T_0}\ds\int_0^t C_{t,s}^2C_s^2\,\E\left[\,\left\|u_1(s)-u_2(s)\right\|_{z,\zeta}^2\,\right]dsdt
\\
&&\leq C_{T_0}\ds\int_0^{T_0} \E\left[\,\left\|u_1(s)-u_2(s)\right\|_{z,\zeta}^2\,\right]ds
\\
&&=C_{T_0}\|u_1-u_2\|^2_{L^2([0,T_0]\times\Omega, H^{z,\zeta})}.
\eeqsn
To the term \eqref{se2} we apply, here below, the fundamental isometry \eqref{isomhilb} to pass from the first to the second line, formula \eqref{battezzata2} of Lemma \ref{lem:weightedpesz2} to pass from the second to the third line, Definition \ref{def:lip} to pass from the third to the fourth line, and finally get:
\beqsn
\ds\int_0^{T_0}&&\hskip-0.3cm\E\left[\left\|\ds\int_0^t E(t,s)(\sigma(s,\cdot,u_1(s))-\sigma(s,\cdot,u_2(s)))dW_s\right\|_{z,\zeta}^2\right]dt
\\
&&=\ds\int_0^{T_0}\ds\int_0^t \E\left[\left\|E(t,s)(\sigma(s,\cdot,u_1(s))-\sigma(s,\cdot,u_2(s)))\right\|_{L_2(\mathcal H_\Xi, H^{z,\zeta})}^2\right]dsdt
\\
\\
&&\lesssim \ds\int_0^{T_0}\ds\int_0^t \E\left[
C^2_{t,s}\|u_1(s)-u_2(s)\|_{H^{z,\zeta}}^2\left(\sup_{\eta\in\Rd}\int_\Rd \frac{\mathfrak{M}(d\xi)}{\langle\xi+\eta\rangle^{2\lambda\mu'}}\right)\right]dsdt
\\
&&\leq \left(\sup_{\eta\in\Rd}\int_\Rd \frac{\mathfrak{M}(d\xi)}{\langle\xi+\eta\rangle^{2\lambda\mu'}}\right)
\ds\int_0^{T_0}\ds\int_0^t C_{t,s}^2\E\left[\left\|u_1(s)-u_2(s)\right\|_{z,\zeta}^2\right]dsdt
\\
&&\leq\left(\sup_{\eta\in\Rd}\int_\Rd \frac{\mathfrak{M}(d\xi)}{\langle\xi+\eta\rangle^{2\lambda\mu'}}\right)
\cdot C_{T_0}\cdot \|u_1-u_2\|^2_{L^2([0,T_0]\times\Omega, H^{z,\zeta})}
\eeqsn
with $C_{T_0}=T_0^{1-2\ell}C_{T_0}$ continuous with respect to $T_0$ since $\ell=\max\{\kappa,\lambda\}<1/2$. Summing up, we have proved that, for some $A>0$,
\beqsn
\|\mathcal{T}u_1-\mathcal{T}u_2\|_{L^2([0,T_0]\times\Omega, H^{z,\zeta})}^2\leq
C_{T_0}\,A\left(1+\sup_{\eta\in\Rd}\int_\Rd \frac{\mathfrak{M}(d\xi)}{\langle\xi+\eta\rangle^{2\lambda\mu'}}\right)
\cdot \|u_1-u_2\|^2_{L^2([0,T_0]\times\Omega, H^{z,\zeta})} \,,
\eeqsn
that is, $\mathcal{T}$ is Lipschitz continuous on $L^2([0,T_0]\times\Omega, H^{z,\zeta})$. Moreover, in view of the assumption \eqref{eq:meascm2} and since
$C_{T_0}$ is continuously dependent on $T_0$ and going to $0$ for $T_0\to0^+$, we can take $T_0>0$ so small that
%
\[
	 C_{T_0}\,A\left(1+\sup_{\eta\in\Rd}\int_\Rd \frac{\mathfrak{M}(d\xi)}{\langle\xi+\eta\rangle^{2\lambda\mu'}}\right)< 1,
\]
and then $\mathcal{T}$ becomes a strict contraction on $L^2([0,T_0]\times\Omega, H^{z,\zeta})$, so that it admits a unique fixed point $u=\mathcal{T}u$,
$u\in L^2([0,T_0]\times\Omega, H^{z,\zeta})$, as claimed.

\vskip+0.2cm

The proof is complete.
\end{proof}

\section{Random-field solutions for the associated linear SPDEs}

We recall that an alternative approach to give meaning to \eqref{eq:SPDE} is the one due to Walsh and Dalang (see \cite{conusdalang,dalang,walsh}), where the stochastic integral in \eqref{eq:mildsolutionSPDE} is defined as a stochastic integral with respect to a martingale measure derived from the random noise $\dot \Xi$. With this alternative approach one obtains a so-called {\em random-field solution}, that is, a solution $u$ defined as a map associating a random variable with each
$(t,x)\in[0,T_0]\times\Rd$, where $T_0>0$ is the time horizon of the equation. In the paper \cite{linearpara} we constructed random-field solutions to linear parabolic SPDEs of the form \eqref{eq:SPDE}. That construction cannot work for non-linear equations of the form \eqref{eq:SPDE}. Indeed, the stationarity condition
$\Lambda=\Lambda(t-s,x-y)$ would be needed, but such condition (fulfilled by SPDEs with constant coefficients) cannot be assumed here, since we want to
deal with general linear operators $L$ in \eqref{eq:SPDE}, that is, admitting variable coefficients, depending on $(t,x)$.
It is well known that, in many cases, the two approaches lead to the same solution $u$ (in some sense) of an SPDE, see \cite{dalangquer} for a precise comparison.
In this final section we compare, in the linear case, the function-valued solution to \eqref{eq:SPDE}, that we constructed in Section 4, with the random-field solution found in \cite{linearpara}. 

Let us then focus on the special case of \eqref{cpsemilin2} with $L$ satisfying \eqref{elle2}, \eqref{eq:sghypoell} with
$\sigma(t,x,u(t,x))=\sigma(t,x)$ and $\gamma(t,x,u(t,x))=\gamma(t,x)$, $\gamma,\sigma\in C([0,T],H^{z,\zeta})$,
$z\ge0$, $\zeta>\frac{d}{2}$. We suppose that $s\mapsto\caF(\sigma)(s)=\nu_s\in L^2([0,T], \caM_b)$, $\caM_b$ the space of
complex-valued measures with finite total variation and that $|\nu_s|_{tv}\in L^\infty(0,T)$. That is, we look at the Cauchy problem
\beqs\begin{cases}\label{cplinbis}
 Lu(t,x) = \gamma(t,x) + \sigma(t,x)\dot{\Xi}(t,x),\quad (t,x)\in(0,T]\times\R^d
 \\
 u(0,x)=u_0(x),\quad x\in\R^d,
\end{cases}
\eeqs
for the linear SPDEs studied in \cite{linearpara}.
Trivially, by Definition \ref{def:lip}, these (more restrictive) hypotheses on $\gamma,\sigma$ imply
$\gamma,\sigma\in \Lip(z,\zeta,r,\rho)\subset \Liploc(z,\zeta,r,\rho)$ for any $r,\rho\ge0$.
If the spectral measure $\mathfrak{M}$ is absolutely continuous with respect to the Lebesgue measure, the condition \eqref{eq:meascm2} simplifies into
\begin{equation}\label{eq:meascm2ac}
  		\exists \lambda\in [0,\onehalf):\quad \int_\Rd \frac{\mathfrak{M}(d\xi)}{\langle\xi\rangle^{2\lambda\mu'}} < \infty.
	\end{equation}
In this case, under condition \eqref{eq:meascm2ac}, Theorem \ref{thm:linearcm} gives the existence of a unique  function-valued solution for the linear
Cauchy problem \eqref{cplinbis}, which we here denote by $u_\fv$.
In Theorem 6 of \cite{linearpara} (see the case (H2) there) we proved that under the assumptions of the present Section there exists also    random-field solution of
\eqref{cplinbis}, which we here denote by $u_\rf$. We now wish to compare $u_\rf$ with $u_\fv$.

\begin{remark}\label{rem:comp}
	Notice that, in analogy with \eqref{eq:rigorous_solution}, $u_\rf$ satisfies (with the usual abuse of notation)
	\beqs\label{eq:sollinbis}
		u_\rf(t,x)&=&v_0(t,x)+\int_0^t\int_{\R^d}\Lambda(t,s,x,y)\gamma(s,y)\,dyds
		+\int_0^t\int_{\R^d}\Lambda(t,s,x,y)\sigma(s,y)\dot{\Xi}(s,y)\,dyds.
	\eeqs
	While the first two terms in the right-hand side of \eqref{eq:sollinbis} clearly coincide with the first two terms
	in the right-hand side of \eqref{eq:rigorous_solution}, the corresponding third, stochastic terms in \eqref{eq:rigorous_solution} and \eqref{eq:sollinbis}
	are defined in different ways.
\end{remark}

In the next Proposition \ref{prop:cmp}, we obtain that a random-field solution of \eqref{cplinbis} is also a function-valued solution. This indirectly provides, via the uniqueness of $u_\fv$, the uniqueness of $u_\rf$.

\begin{proposition}\label{prop:cmp}
	Consider the linear Cauchy problem \eqref{cplinbis} with $L$ as in \eqref{elle2}, \eqref{eq:sghypoell},
	$\gamma,\sigma\in C([0,T],H^{z,\zeta})$, $z\ge0$, $\zeta>\frac{d}{2}$, $s\mapsto\caF(\sigma)(s)=\nu_s\in L^2([0,T], \caM_b)$,
	$|\nu_s|_{tv}\in L^\infty(0,T)$. Suppose that $\mathfrak{M}$ is absolutely continuous and \eqref{eq:meascm2ac} holds.
	Let $u_\rf$ and $u_\fv$ be the random-field solution
	and the function-valued solution of \eqref{cplinbis}, respectively.
	Then, $u_\rf=u_\fv=u$.
\end{proposition}
The argument of the proof is based on \cite[Proposition 3.12]{dalangquer}, and is the same one employed to prove the analogous result given in
\cite[Proposition 4.17]{ACS19b}. 
%


\end{document}